\documentclass{amsart}[12pt]
\usepackage{amsthm}
\usepackage{amsfonts}
\usepackage{comment}
\usepackage{amssymb}
\usepackage{eufrak}
\usepackage{comment}
\usepackage{hyperref}
\usepackage{mathtools}
\usepackage{eucal}
\numberwithin{equation}{section}
\usepackage{breqn}
\allowdisplaybreaks
\usepackage[OT2,T1]{fontenc}
\usepackage{cleveref}
\usepackage[shortlabels]{enumitem}

\DeclareSymbolFont{cyrletters}{OT2}{wncyr}{m}{n}\DeclareMathSymbol{\Sha}{\mathalpha}{cyrletters}{"58}
\def\legendre(#1,#2){%
\begin{pmatrix}
#1\cr 

#2\cr
\end{pmatrix}} 
\newtheorem{theorem}{Theorem}
\newtheorem{Theorem}{Theorem}[section]
\newtheorem{Corollary}[Theorem]{Corollary}
\newtheorem{Lemma}[Theorem]{Lemma}
\newtheorem{Proposition}[Theorem]{Proposition}

\theoremstyle{definition}

\theoremstyle{remark}
\newtheorem*{remark}{Remark}
\usepackage[margin=1.25in]{geometry}
\usepackage{cite}
\title{Quasimodular moonshine and arithmetic connections}
\begin{document}

\author[L. Beneish]{Lea Beneish}
\address{Department of Mathematics, Emory University}
\email{lea.beneish@emory.edu}

\begin{abstract} We prove the existence of a module for the largest Mathieu group, whose trace functions are weight two quasimodular forms. Restricting to the subgroup fixing a point, we see that the integrality of these functions is equivalent to certain divisibility conditions on the number of $\mathbb{F}_p$ points on Jacobians of modular curves. Extending such expressions to arbitrary primes, we find trace functions for modules of cyclic groups of prime order with similar connections. Moreover, for cyclic groups, we give an explicit vertex operator algebra construction whose trace functions are given only in terms of weight two Eisenstein series.%The module constructions are in terms of Heisenberg and Clifford module vertex algebras.
\end{abstract}
\maketitle
\section{Introduction}
The idea of moonshine dates back to the $1970$s when McKay noticed that $196884$, the coefficient of $q$ in the expansion of the normalized $j$-invariant, $J(\tau)$, is the sum of $1$ and $196883$, which are dimensions of irreducible representations of the monster group $\mathbb{M}$.  Thompson further conjectured that this is true for the rest of the coefficients of $J(\tau)$ --- that there exists a graded infinite-dimensional $\mathbb{M}$-module \[ V^\natural = \bigoplus_n V_n^\natural \] such that $\sum\limits_{n=-1}^\infty \dim(V_n^\natural)q^n =J(\tau).$ 

More generally, Thompson conjectured that there exists a module $V^\natural$ such that the graded trace functions of other group elements $g\in \mathbb{M}$, defined as $T_g(\tau):=\sum_{n=-1}^\infty \text{tr}(g\mid V_n^\natural)q^n$, are distinguished functions \cite{Thompson, MR554402}. In $1979$, Conway and Norton conjectured further (the Monstrous Moonshine Conjecture) that for each $g\in \mathbb{M}$ the graded trace function $T_g$ is the unique modular function that generates the genus zero function field arising from a specific subgroup $\Gamma_g$ of $SL_2({\mathbb{R}})$, normalized such that $T_g(\tau)=q^{-1}+O(q)$ \cite{ConwayNorton}. In $1988$, Frenkel, Lepowsky, and Meurman \cite{FLM} constructed $V^\natural$ and in $1992$, Borcherds proved the Monstrous Moonshine Conjecture, showing that $V^\natural$ is the moonshine module with properties conjectured by Conway and Norton \cite{Borcherds}.

Almost twenty years after the proof of the Monstrous Moonshine Conjecture, other sporadic simple groups appeared in the theory of moonshine. %, certain mock modular forms appeared in the theory of moonshine. These forms were conjectured to be graded trace functions associated to other sporadic simple groups. % the holomorphic parts of harmonic Maass forms, appeared in the theory of moonshine. Certain mock modular forms were conjectured to be the trace functions for graded modules for other sporadic simple groups. One of these groups, the largest Mathieu group, $M_{24}$,
% is the permutation automorphism group of the binary Golay code. It 
One of these groups, $M_{24}$, entered the story in $2010$ when physicists Eguchi, Ooguri, and Tachikawa \cite{EOT} noticed that the dimensions of representations of $M_{24}$ are the multiplicities (the first few coefficients of a function denoted $\Sigma(\tau)$) of superconformal algebra characters in the K$3$ elliptic genus. This led to the conjecture of the existence of an infinite-dimensional graded $M_{24}$-module whose existence was proven in $2012$ by Gannon \cite{Gannon}. More precisely, Gannon showed that there exists a graded $M_{24}$-module \[K=\bigoplus\limits_{n} K_{n-1/8}\] whose graded dimension function, $H_e(\tau):=-2q^{-1/8}+\sum\limits_{n=1}^{\infty} \text{tr}(e\mid K_{n-1/8})$  is equal to $\Sigma(\tau)$ from the expansion of the K$3$ elliptic genus in \cite{EOT}. The graded trace functions  for other $g\in M_{24}$, defined analagously and denoted $H_g(\tau)$, are mock modular forms of weight $1/2$ and level $|g|$, the order of $g$. This is in contrast to monstrous moonshine, whose trace functions are weakly holomorphic modular functions.

This $M_{24}$ moonshine belongs to a larger class of $23$ moonshines, umbral moonshine, which was conjectured by Cheng, Duncan, and Harvey \cite{CDH} and whose existence was proven by Duncan, Griffin, and Ono \cite{DGO}. But there are some disparities between umbral moonshine and monstrous moonshine. In addition to the fact, already mentioned, that umbral moonshine involves (mock modular) forms that are not of integral weight, another difference compared to the monstrous case is that despite some recent progress by Anagiannis--Cheng--Harrison, Cheng--Duncan, Duncan--Harvey, and Duncan--O'Desky (see \cite{DH, DO, CDMeroJacobi,AnagiannisChengHarrison}), the umbral moonshine theory does not yet include module constructions in all of its cases.

An initial motivation for this work was to bring umbral moonshine in a closer context to monstrous moonshine. In carrying this out we discovered two new infinite families of similar phenomena: one with arithmetic content, and one for which we are able to give vertex algebra constructions.

 We start with Mathieu moonshine and reframe $M_{24}$ moonshine in terms of trace functions that are weight two quasimodular forms instead of mock modular forms. Restricting to $M_{23}$ we find expressions for these forms that contain arithmetic information. We generalize this type of expression to $\mathbb{Z}/ N \mathbb{Z}$ for arbitrary $N$ prime, from which we can observe more connections of a similar kind.  At the expense of arithmetic connections, we give a second set of quasimodular trace functions for another $\mathbb{Z}/ N \mathbb{Z}$-module which are given only in terms of Eisenstein series. We construct these modules explicitly as tensor products of Heisenberg and Clifford module vertex operator algebras.

%To reframe $M_{24}$ moonshine in the way discussed above, we first modify the trace functions of Mathieu moonshine so that they keep essentially the same representation-theoretic content, but are more similar to those of monstrous moonshine. Since our specific aim is to find integral weight modular forms associated to $M_{24}$ as opposed to half-integral weight mock-modular forms, we take the original $H_g(\tau)$ and multiply by $\eta^3(\tau)$ and the result is a mixed mock modular form of weight two. A natural tool to use when moving from spaces of holomorphic mixed mock modular forms to modular forms is holomorphic projection. In our case, since the weight is two, the result will be quasimodular.
The quasimodular forms that we give as trace functions of an $M_{24}$-module come from multiplying the functions $\hat{H}_g(\tau)$, the completions of the mock modular forms $H_g(\tau)$, by $\eta^3(\tau)$ to bring the weight to $2$, and then taking the holomorphic projection. In this way we define weight two quasimodular forms $Q_g(\tau)$ for every $g\in M_{24}$. The first example of a weight two quasimodular form is \[E_2(\tau)=1-24\sum\limits_{n=1}^{\infty} \sigma_1(n) q^n,\] the usual weight $2$ Eisenstein series. In our setting, each coefficient of $-2E_2(\tau)$ has a natural interpretation as the dimension of a virtual $M_{24}$-module. This is because our $Q_e(\tau)$ is equal to $-2E_2(\tau)$ and we prove the following:

 \begin{theorem}
 There exists a virtual graded $M_{24}$-module $V=\bigoplus\limits_{n}V_n$ such that \[Q_g(\tau)=\sum\limits_{n=0}^{\infty}\emph{tr}(g\mid V_n)q^n.\] 
\end{theorem}

Interestingly, if one restricts to the trace functions of $M_{23}$, a subgroup of $M_{24}$ whose group elements have fixed points in their permutation representations, the $Q_g(\tau)$ have convenient expressions in terms of Eisenstein series and cusp forms. These expressions are of the form 
\begin{equation}\label{eiscusp} Q_g^{}(\tau)=-2E_{2,N}(\tau)+\frac{N}{n_{N}}G_N(\tau)\end{equation}
where we use $N$ to denote the order of $g$, we let $n_N:= \text{num}\left(\dfrac{N-1}{12}\right)$, $E_{2,N}(\tau)$ is an expression in terms of Eisenstein series (cf. \eqref{e2n}), and $G_N$ is a specific cusp form of weight $2$ for $\Gamma_0(N)$ with integer coefficients.

The trace functions of this module involve weight two cusp forms and so they contain arithmetic information. As an example, for certain $N$, the Jacobian of $X_0(N)$, denoted $J_0(N)$, is an elliptic curve, and so the integrality of the trace functions is equivalent to certain divisibility conditions on the number of $\mathbb{F}_p$ points on these curves (this holds even when the dimension of $J_0(N)$ is greater than $1$, but we restrict to elliptic curves here for simplicity). These results depend on the cooperation of Eisenstein series and cusp forms to sum to integral coefficients. For example we have the following (known, see Appendix A of \cite{Katz}) divisibility conditions arising from $M_{23}$:
% \begin{Corollary}\label{cor}
%\begin{enumerate} 
%\item []
% \item  $p+1\equiv 0 \mod{5}$ if and only if $5\mid \# J_0(11)(\mathbb{F}_p)$,
% \item $p+1\equiv 0 \mod{3}$ if and only if $3\mid \# J_0(14)%(\mathbb{F}_p)$,
 %\item For $p\neq 2$, we have that $p+1\equiv 0 \mod{4}$ if and only %if $4\mid \# J_0(15)(\mathbb{F}_p)$.
 
%\end{enumerate}
%\end{Corollary}
\begin{Corollary}\label{cor} 
\begin{enumerate} 
\item []
 \item  For $p\neq 11$, we have $5\mid \# J_0(11)(\mathbb{F}_p)$. %and we have $ \# J_0(11)(\mathbb{F}_{11}) \equiv 1 \pmod{5}$.
 \item For $p\neq 2,7$ we have $3\mid \# J_0(14)(\mathbb{F}_p)$. %and we have $ \# J_0(14)(\mathbb{F}_{2}) \equiv 2 \pmod{3}$ and $ \# J_0(14)(\mathbb{F}_{7}) \equiv 1 \pmod{3}$.
 \item For $p\neq 3,5$, we have $4\mid \# J_0(15)(\mathbb{F}_p)$. %and we have $\#J_0(15)(\mathbb{F}_3)\equiv 1 \pmod{4}$ and $\#J_0(15)(\mathbb{F}_5)\equiv 3 \pmod{4}$.
 
\end{enumerate}
\end{Corollary}

Moreover, the pattern we observe in the denominators of the cusp forms reflects a result of Mazur on a congruence between Eisenstein series and cusp forms in the cases where $N$ is prime \cite{Mazur}. The type of expression in (\ref{eiscusp}) can be generalized, and in fact, the formula for $Q_g^{}(\tau)$ does not depend on $M_{23}$ and can be defined for arbitrary $N$. We restrict our focus to $N$ prime and prove the existence of a $\mathbb{Z}/ N \mathbb{Z}$-module with trace functions: 
 
\begin{equation}\label{znztrace}f_{g}^{(N)}(\tau):=\begin{dcases}
   -\frac{\ell_N}{n_N}E_{2}(\tau) & \text{if } g=e, \\
    -\frac{\ell_N}{n_N}E_{2,N}(\tau)-\frac{N}{n_N}G_N(\tau) &\text{if } g\neq e,
\end{dcases} \end{equation}
where $\ell_N:=\text{num}\left(\dfrac{N^2-1}{24}\right)$. The functions $f_{g}^{(N)}(\tau)$ are quasimodular forms of weight $2$ with integral coefficients defined in terms of Eisenstein series and cusp forms. The $E_{2,N}(\tau)$ (cf. \eqref{e2nprime}) are again defined in terms of Eisenstein series and the $G_N(\tau)$ are certain cusp forms of level $N$ and weight $2$ with integer coefficients (cf. Section \ref{Section3}, \Cref{eiscuspcong}). With these definitions we state the following result:

\begin{theorem}\label{thm2} Let $N$ be a prime and $f_g^{(N)}(\tau)$ be as in (\ref{znztrace}). Then there exists a virtual graded $\mathbb{Z}/N\mathbb{Z}$-module $V^{(N)}=\bigoplus\limits_{n}V^{(N)}_n$ such that \[f_{g}^{(N)}(\tau)=\sum\limits_{n=0}^{\infty}\emph{tr}(g \mid V^{(N)}_n)q^n.\] 
\end{theorem}

As a consequence, for $N$ prime, one can observe many more examples analogous to those in Corollary \ref{cor} arising from the trace functions $f^{(N)}_g(\tau)$.

In the trace function for $\mathbb{Z}/ N \mathbb{Z}$ in Theorem \ref{thm2}, the multiple in front of $E_{2,N}(\tau)$ is $\dfrac{\ell_N}{n_N}$ (recall that $\ell_N=\text{num}\left(\frac{N^2-1}{24}\right)$ and $n_N=\text{num}\left(\frac{N-1}{12}\right)$). We have that $\ell_N$ is the minimal number which clears the denominators of $E_{2,N}(\tau)$, and we can further divide by $n_N$ because we find a cusp form that satisfies a congruence modulo $n_N.$ If we restrict our functions to be only in terms of Eisenstein series and do not use cusp forms, we can no longer divide by $n_N$ and instead have trace functions as follows:
\noindent for $N$ prime, let
\begin{equation}\label{Feq}F_{g}^{(N)}(\tau):=\begin{dcases}
   -\ell_NE_{2}(\tau) & \text{if } g=e, \\
 -\ell_NE_{2,N}(\tau) &\text{if } g\neq e.
\end{dcases} \end{equation}
These are weight two purely Eisenstein quasimodular trace functions for a $\mathbb{Z}/ N \mathbb{Z}$-module. Although this comes at the cost of arithmetic connections arising from cusp forms, we give a vertex operator algebra construction for this module. For each $N$, this module is $W^{(N)}_{\text{tw}}$, a twisted module for the vertex operator algebra which we denote $W^{(N)}$. The vertex operator algebras $W^{(N)}$ are defined as tensor products of Heisenberg and Clifford module vertex algebras. The precise construction of $W^{(N)}$ is given in Section \ref{Section5}.

\begin{theorem} For $N$ prime, the canonically twisted module $W_{\emph{tw}}^{(N)}=\bigoplus\limits_{n}W^{(N)}_{\emph{tw},n}$ of the vertex operator algebra $W^{(N)}$ is an infinite dimensional virtual graded module for $\mathbb{Z}/ N \mathbb{Z}$  such that \[F_{g}^{(N)}(\tau)=\sum\limits_{n=0}^{\infty}\emph{tr}(g\mid W^{(N)}_{\emph{tw},n})q^n.\] 
\end{theorem}

%\widetilde{A}(\mathfrak{p})_{\text{tw}}$ is a certain tensor product of  $ A(\mathfrak{p})_{\text{tw}}^+$ defined in \cite{DH} and $V^{(N)}$ is a  Heisenberg vertex algebra for an $(N^2-1)$-dimensional complex vector space (c.f. $\S \, 5$).

This paper is organized as follows. Section \ref{Section2} includes information on the composition of the quasimodular forms we associate to $M_{24}$ including a proof that they are quasimodular and a proof of the existence of the graded module for which they are the trace functions. Section \ref{Section3} gives an alternative expression for the quasimodular forms for the cases $g\in M_{23}$, generalizes that expression to $\mathbb{Z}/ N \mathbb{Z}$ for arbitrary $N$ prime, and includes the proof of the existence of the corresponding $\mathbb{Z}/ N \mathbb{Z}$ module. Section \ref{Section3} additionally includes an elementary proof of Mazur's congruence between Eisenstein series and cusp forms. Section \ref{Section4} explains the arithmetic connections of the trace functions for the modules in Section \ref{Section3}. Section \ref{Section5} includes purely Eisenstein quasimodular trace functions for $\mathbb{Z}/ N \mathbb{Z}$ and constructs the corresponding modules explicitly.
 
\subsection*{Acknowledgments} The author is grateful to John Duncan for suggesting the topic and for his valuable advice and guidance throughout the process. The author would like to thank Victor Manuel Aricheta, Jeffrey Lagarias, Kimball Martin, Michael Mertens, Jackson Morrow, Ken Ono, Hanson Smith, and David Zureick-Brown for helpful comments and discussions. The author is also very grateful to Jeremy Rouse for spotting an error in an earlier version.

\section{Quasimodular $M_{24}$ forms}\label{Section2}
The setup of the $M_{24}$ functions starts with building forms of weight $2$ from the original Mathieu moonshine functions. More specifically, the forms come from multiplying the completions $\hat{H}_g(\tau)$ of the mock modular forms $H_g(\tau)$ from Mathieu moonshine by $\eta^3(\tau)$. Note that $\hat{H}_g(\tau)\eta^3(\tau)$ is a non holomorphic modular form of weight $2$, and since it does not have singularities at cusps, we can apply holomorphic projection to extract something holomorphic. In weight $2$, holomorphic projection results in a quasimodular form. For more information about mock modular forms, their completions, and holomorphic projection, we refer the reader to \cite{Kensbook}.
 
To give the expression for the holomorphic projection explicitly, we first define a function $F^{}_2(\tau)$ as follows:
 \[F^{}_2(\tau):=\sum\limits_{\substack{r>s>0\\ r-s \text{ odd}}}sq^{rs/2}.\] The holomorphic projections of the $\hat{H}_g(\tau)\eta^3(\tau)$ are given in terms of $H_g(\tau)\eta^3(\tau)$ and some multiple of $F_2(\tau)$ for each $g\in M_{24}$. We will see that these expressions are quasimodular forms. 
 %More precisely a theorem of Mertens (see \cite{Michael}) gives that the holomorphic projection of this mixed mock modular form $H_e(\tau)\eta^3(\tau)$ is $H_e(\tau)\eta^3(\tau)-48F^{}_2(\tau)$, and that this is a quasimodular form of weight $2$ and level $1$, so is some multiple of $E_2(\tau)$. This theorem of Mertens holds more generally and can in fact be specialized to
 For the first case, Dabholkar, Murthy, and Zagier give a formula in \cite{DMZ} that (when rearranged) says
 \begin{equation}\label{dmz} H_e(\tau)\eta^3(\tau)-48F_2(\tau)=-2E_2(\tau).\end{equation}
 \begin{remark} Dabholkar, Murthy, and Zagier also define a higher weight analogue of $F_2^{(k)}$ for $k\geq 2$. For our purposes $k=2$.\end{remark}

A theorem of Mertens, when specialized to these functions, gives explicitly that the holomorphic projection of $\hat{H}_e(\tau)\eta^3(\tau)$ is equal to the left hand side of (\ref{dmz}) and is a quasimodular form \cite{Michael}. Mertens' theorem can be applied to the other functions $H_g(\tau)$, and %More precisely a theorem of Mertens (see \cite{Michael}) gives that the holomorphic projection of this mixed mock modular form $H_e(\tau)\eta^3(\tau)$ is $H_e(\tau)\eta^3(\tau)-48F^{}_2(\tau)$, and that this is a quasimodular form of weight $2$ and level $1$, so is some multiple of $E_2(\tau)$
in fact, we have such a formula more generally, for any $g\in M_{24}$. Let $\chi(g)$ be the number of fixed points of $g$ in the $24$-dimensional permutation representation of $M_{24}$. We define \[Q_g(\tau):=H_{g}(\tau)\eta^3(\tau)-2\chi(g)F^{}_2(\tau).\] This is consistent with the $g=e$ case given above because the number of fixed points in that case is $\chi(e)=24$. We will show that these functions are quasimodular. 

To be precise in the description of the $Q_g(\tau)$, we first define $\rho_g$, a function from $\Gamma_0(|g|)$ to $\mathbb{C}$ given by $\rho_g(\gamma):=\text{exp}\left(2\pi i\left(-\dfrac{cd}{|g|h}\right)\right)$, where $h$ is the minimal length among cycles in the cycle shape of $g$ and $c,d$ are the entries of the lower row of a matrix $\gamma$ in $\Gamma_0(|g|)$. 

%\begin{Theorem}[Mertens]
%Suppose $g(\tau)$ and $h(\tau)$ are weight $3/2$ theta functions in the spaces $S_{3/2}(\Gamma, \nu_g)$ and $S_{3/2}(\Gamma, \nu_h)$ respectively,  with Fourier expansions 
%\[ g(\tau):=\sum_{i=1}^s \sum_{n\in \mathbb{Z}} n \chi_i(n)q^{n^2} \]
%\[ h(\tau):=\sum_{j=1}^t \sum_{n\in \mathbb{Z}} n \psi_j(n)q^{n^2} \]
%where each $\chi_i$ and $\psi_i$ is a Dirichlet character. Moreover, suppose $h(\tau)$ is the shadow of a weight $1/2$ harmonic Maass form $f(\tau)\in H_{1/2}(\Gamma, \overline{\nu}_h)$. Define the function 
%\[D^{f,g}(\tau)=2\sum\limits_{r=1}^{\infty}\sum\limits_{\chi_i, \psi_j}\sum\limits_{\substack{m, n> 0\\ m^2-n^2=r}} \chi_i(m)\overline{\psi}_j(n) (m-n)q^{r}.\]
%If $f(\tau)g(\tau)$ has no singularities at cusps, then $f^+(\tau)g(\tau)+D^{f,g}(\tau)$ is a quasimodular form.
%\end{Theorem}
%\vspace{3mm}
%\noindent Applying this theorem, we get the following proposition:\\
\begin{Proposition}  % \begin{enumerate}
%\item []
The $Q_g(\tau)$ are quasi-modular of weight $2$ on $\Gamma_0(|g|)$, with multiplier system $\rho_g$. 
%\item The $Q_g(\tau)$ have expansions starting with $-2$ at infinity and starting with an explicitly computed multiple of $q$ at other cusps.
%\end{enumerate}
\end{Proposition}
\begin{proof} The following explicit formula 
\begin{equation}H_g(\tau)=\dfrac{\chi(g)}{24}H_e(\tau)- \dfrac{T(g)}{\eta^3(\tau)}. \end{equation}
was obtained in \cite{2,3,4,5} (see Section 3 of \cite{holographic}). When rearranged, the above formula relates $H_e(\tau)\eta^3(\tau)$ to each of the $H_g(\tau)\eta^3(\tau)$. The $T(g)$ are weight $2$ forms on $\Gamma_0(|g|)$ with multiplier $\rho_g$ and their explicit expressions are given in Appendix $B.3.1$ of \cite{DGO}. Combining these with the equation (\ref{dmz}) for $H_e(\tau)\eta^3(\tau)$ gives that the functions $H_g(\tau)\eta^3(\tau)$ are quasimodular of weight $2$.
\end{proof}

\noindent Now that we have described our new functions $Q_g(\tau)$ we show that there exists an $M_{24}$-module for which these are the graded trace functions.
 \begin{Theorem}
 There exists a virtual graded $M_{24}$-module $V=\bigoplus\limits_{n}V_n$ such that \[Q_g(\tau)=\sum\limits_{n=0}^{\infty}\emph{tr}(g\mid V_n)q^n.\]

\end{Theorem}
\begin{proof} First we show that the $Q_g(\tau)$ have integral coefficients. Gannon \cite{Gannon} shows that the functions $H_{g}(\tau)$ have integral coefficients. It is known that $\eta^3(\tau)$ has integral coefficients, and $F_2(\tau)$ also has integral coefficients. So we know that\[Q_g(\tau):=H_{g}(\tau)\eta^3(\tau)-2\chi(g)F^{}_2(\tau)\] 
 must have integral coefficients.

Next we show that the multiplicities $\text{m}_i(n)$ of the $M_{24}$ irreducible representations in the class functions defined by the coefficients of $Q_g(\tau)$ are integral. 

Gannon shows that the multiplicity generating function \begin{equation}\label{gannonmi}\sum\limits_{n>0} m_i(n) q^n=\dfrac{1}{|M_{24}|}\sum\limits_{g\in M_{24}} H_g(\tau)\overline{\chi_i(g)}\end{equation} (with $\chi_i$ an irreducible character of $M_{24}$) has integral coefficients.
What we need to show is that the coefficients $\text{m}_i(n)$ are integral, where
\begin{equation}\label{mi}\sum\limits_{n>0} \text{m}_i(n) q^n=\dfrac{1}{|M_{24}|}\sum\limits_{g\in M_{24}} \left[H_g(\tau)\eta^3(\tau)-2\chi(g)F^{}_2(\tau)\right]\overline{\chi_i(g)}.\end{equation}
To do this, we can split the right hand side of equation (\ref{mi}) into two parts. First consider $\dfrac{1}{|M_{24}|}\sum\limits_{g\in M_{24}} H_g(\tau)\eta^3(\tau)\overline{\chi_i(g)}$. This differs from (\ref{gannonmi}) only from multiplying by $\eta^3(\tau)$, which does not change the integrality. So it suffices to show that $\dfrac{1}{|M_{24}|}\sum\limits_{g\in M_{24}} \chi(g) F^{}_2(\tau)\overline{\chi_i(g)}$ has integral coefficients.
%To see this, we note that $\chi(g)$ can be expressed as $\chi_1(g)+\chi_2(g)$ from the character table of $M_{24}$, so we write
This is the same as showing that \[F^{}_2(\tau)\dfrac{1}{|M_{24}|}\sum\limits_{g\in M_{24}} \chi(g)\overline{\chi_i(g)}=F^{}_2(\tau)\langle \chi, \chi_i \rangle\] has integral coefficients. We already know that $F^{}_2(\tau)$ has integral coefficients. The integrality of $\langle \chi, \chi_i \rangle$ can be seen from the fact that $\chi(g)$ is a character of a module, and so $\langle \chi, \chi_i \rangle$ is the multiplicity of $\chi_i$ in $\chi$, which is necessarily integral. Thus the $\text{m}_i(n)$ from (\ref{mi}) are integral.

% \[\dfrac{1}{|M_{24}|}\sum\limits_{g\in M_{24}} F^{}_2(\tau) \left(\chi_1(g)+\chi_2(g)\right) \overline{\chi_i(g)}\] which is equal to $\langle F^{}_2(\tau)(\chi_1+\chi_2), \chi_i\rangle$ and is equal to $F^{}_2(\tau)$ if $i=1,2$ and $0$ otherwise. Therefore the generating function (\ref{mi}) has integral coefficients and thus the multiplicities are integral.
\end{proof}
%\noindent When restricted to prime index, these functions are asymptotically Hecke eigenforms so it is natural to ask about the limiting distribution over prime levels. The result is as follows:

%\begin{Proposition} As $p\to \infty$ the $V_p$ tend to direct sums of copies of a virtual representation with certain explicitly computed multiplicities, all but four of which are negative. 

%\subsection{Restriction to $M_{23}$}

%Dabholkar, Murthy, and Zagier give the following formula relating the $H_e(\tau)$ to $F^{(2)}_2(\tau)$, weight $2$ Eisenstein series, and $\eta^3(\tau)$:

%\[H_e(\tau)=\dfrac{24F_2^{(2)}(\tau)-E_2(\tau)}{\eta^3(\tau)} \]%This module is very virtual, the first few coefficients of $H_e(\tau)\eta^3(\tau)$ we have:
%\begin{align*}
 %96&=45+45+1+1+1+1+1+1\\
 %192&=231+231-45-45-45-45-45-45\\
 %144&=770+770-231-231-231-231-231-231-1-1-1-1-1-1-1-1-1-1\\
 %& \vdots
 %\end{align*}
 
 \section{More general framework}\label{Section3}
In this section we show that for certain conjugacy classes $[g]$, the $Q_g(\tau)$ have convenient expressions containing arithmetic information. Further, we show that this type of expression can be generalized to be in terms of an arbitrary prime $N$. \\
If we restrict to $M_{23}$, a subgroup of $M_{24}$ for which all $g$ have $\chi(g)\neq 0$, we can give an alternate expression for the corresponding $Q_g(\tau)$. First, let \begin{equation}\label{e2n}E_{2,N}(\tau):=\dfrac{1}{i(N)\varphi(N)}\sum\limits_{M \mid N} \mu\left(\frac{N}{M}\right)M^2E_2(M\tau)\end{equation} where $N$ is defined to be the order of $g$, $i(N)$ is the index of $\Gamma_0(N)$ in $SL_2(\mathbb{Z})$, and $\varphi$ is the Euler totient function. Note that these functions $E_{2,N}(\tau)$ are quasimodular of weight $2$ on $\Gamma_0(N)$.

 For each $N=|g|$ with $g\in M_{23}$, let $G_N(\tau)$ denote the specific cusp form of level $N$ given explicitly in the appendix. We also let $n_{N}:=\text{num}\left(\dfrac{N-1}{12}\right)$. Then we have the following formula:
\[H_{g}(\tau)\eta^3(\tau)-2\chi(g)F^{}_2(\tau)=-2E_{2,N}(\tau)+\frac{N}{n_{N}}G_N(\tau). \]
This can be easily checked by comparing case-by-case to formula ($B.24$) in \cite{DGO}.
 From this it follows that for $g\in M_{23}$
\begin{equation}\label{eiscuspmathieu}Q_g^{}(\tau)=-2E_{2,N}(\tau)+\frac{N}{n_{N}}G_N(\tau).\end{equation}

Note that the formula for $Q_g(\tau)$ above is defined in terms of $N=|g|$ but the expression does not depend on $M_{23}$ at all. We can define such functions $Q_N(\tau)$ for arbitrary prime $N$ with suitable cusp forms $G_N(\tau)$ that come from a result of Mazur. In what follows we show precisely how to define the $Q_N(\tau)$ including how to use the result of Mazur to determine the cusp form.

The multiples of the cusp forms $G_N(\tau)$ in the expressions (\ref{eiscuspmathieu}) have denominators equal to $n_N$ (recall, $n_N=\text{num}\left(\frac{N-1}{12}\right)$). In the trace functions where $N$ is prime, these denominators reflect a result of a congruence between Eisenstein series and cusp forms that is due to Mazur. %A natural question to ask then is whether there are other groups with modules that have meaningful quasimodular weight $2$ forms as their trace functions. We give an affirmative answer for $\mathbb{Z}/{N\mathbb{Z}}$ for $N$ an arbitrary prime by proving the existence of a $\mathbb{Z}/{N\mathbb{Z}}$-module with trace functions of this type. 
Due to this result, we are able to define quasimodular forms as in (\ref{eiscuspmathieu}) and show the existence of a $\mathbb{Z}/{N\mathbb{Z}}$-module such that these forms are its trace functions.

First, we give a more elementary argument for Mazur's result (Proposition $5.12$ of \cite{Mazur}) that there exists a cusp form congruent to the Eisenstein series of level $N$ for $N$ prime.  
Denote the $m$th coefficient of the normalized Eisenstein series $\dfrac{1}{24}(NE_2(N\tau)-E_2(\tau))$ by $\sigma_N(m)$. To prove the existence of a cusp form whose coefficients are congruent to $\sigma_N(m)$, we will use theta series, defined in \cite{Gross}. Let $i,j \in \{1\dots n\}$ where $n$ is the number of left ideal classes in the quaternion algebra over $\mathbb{Q}$ ramified at the two places $N$ and $\infty$.  Then the theta series are defined as \begin{equation}\label{thetaseries}f_{ij}:=\frac{1}{2w_j}+\sum\limits_{m\geq1} B_{ij}(m) q^m,\end{equation} where $B_{ij}(m)$ are entries of the Brandt matrix of degree $m$, and $w_j$ are integers that correspond to the cardinalities of certain groups. We will use the fact that these $f_{ij}$ are functions with integral coefficients (except the constant term) on the upper half plane. In fact, the $f_{ij}$ span $M_2(\Gamma_0(N))$ and a certain explicit linear combination of $f_{ij}$ recovers the normalized Eisenstein series. 
See \cite{Quat} and \cite{Martin} for related results using these theta series. We will use them to prove the following proposition.
 
\begin{Proposition}\label{eiscuspcong} Let $N$ be prime. Then there exists a cusp form  $g(\tau)=\sum\limits_{m>0} c_g(m)q^m$ of weight $2$ on $\Gamma_0(N)$ with integer coefficients such that \[c_g(m)\equiv \sigma_N(m)\pmod{n_N} \] for all $m>0$.

\end{Proposition}

\begin{proof}
First we treat the cases $N=2,3$. For these cases, $n_N=1$ and the space of weight two cusp forms on $\Gamma_0(N)$ is empty. This means $c_g(m)=0$ for all $m$, and since $\sigma_N(m)$ are necessarily integers, the statement $0\equiv \sigma_N(m) \pmod{1}$ is true for all $m$.\\
For the rest of the cases, we use that any modular form of weight $2$ can be written as a linear combination of the $f_{ij}$, defined in (\ref{thetaseries}). We would like to find a linear combination of $f_{ij}$ with constant term zero, that is, a cusp form, which is congruent to the Eisenstein series modulo $n_N$. Note that finding such a linear combination with constant term zero ensures we have a cusp form. This is because in weight $2$, the sum of the constant terms of a modular form at all cusps must be zero. Since we have that $N$ is prime, $\Gamma_0(N)$ has only two cusps, so it is sufficient to check the vanishing at one cusp.
Recall (equation (5.7) in \cite{Gross}) that the Eisenstein series $\dfrac{1}{24}(NE_2(N\tau)-E_2(\tau))$ is given by \[  \sum\limits_{j=1}^n f_{ij}=\frac{N-1}{24}+\sum\limits_{m>0} \sigma_N(m) q^m =\frac{N-1}{24}+\sum\limits_{m>0}\sum\limits_{j=1}^{n}B_{ij}(m)q^m, \]
for any $i\in\{1\dots n\}$.

\noindent We will use the fact that all primes $N> 3$ are $1,5,7, 11 \pmod{12}.$ \\

\noindent The first case is $N\equiv1\pmod{12}$ which implies that $N-1\equiv 0\pmod{12}$ so $12\mid N-1$.

Here we take \[g(\tau)=\left(\sum\limits_{j=1}^n f_{1j}\right)-w_1f_{11} n_N. \]

The constant term of $2w_1f_{11}$ is $1$ and so the constant term of $-w_1f_{11} n_N$ is $-n_N\frac{1}{2}=-\frac{N-1}{24}$ since $12\mid N-1$, $(N-1,12)=12$. Thus we have:
\[g(\tau)= \frac{N-1}{24} + \sum\limits_{m>0} \sum\limits_{j=1}^{n}B_{1j}(m) q^m - \frac{N-1}{24} -{w_1} {n_N}\sum\limits_{m>0} B_{1j}(m)q^m, \]
and we can cancel the constant terms to get
\[ g(\tau)=\sum\limits_{m>0} \sum\limits_{j=1}^{n}B_{1j}(m) q^m -w_1{n_N}\sum\limits_{m>0} B_{1j}(m)q^m \equiv \sum\limits_{m>0} \sum\limits_{j=1}^{n}B_{1j}(m) q^m \pmod{n_N}, \]
where the left hand side has integer coefficients and the right hand side is equal to the normalized Eisenstein series minus its constant term.\\

\noindent The next case is $N\equiv5\pmod{12}$ so $(N-1,12)=4$. In particular, $N-1$ is coprime to $3$.%which implies that $N-1\equiv 4\pmod{12}$ so $4\mid N-1$, which means that $N-1=12k+4=4(3k+1)$, thus $3\nmid N-1$ and so $6,12 \nmid N-1$, which means $(N-1,12)=4$.

Here we take \[g(\tau)=\left(3\sum\limits_{j=1}^n f_{1j}\right)-w_{1}f_{11} n_N. \]

The constant term of $2w_1f_{11}$ is $1$ and so the constant term of $-w_1f_{11} n_N$ is $-n_N\frac{1}{2}=-\frac{3(N-1)}{24}$ since $(N-1,12)=4$. Thus we have:
\[g(\tau)= \frac{3(N-1)}{24} + 3\sum\limits_{m>0} \sum\limits_{j=1}^{n}B_{1j}(m) q^m - \frac{3(N-1)}{24} -w_1 {n_N}\sum\limits_{m>0} B_{1j}(m)q^m, \]
and we can cancel the constant terms to get
\[g(\tau)= 3\sum\limits_{m>0} \sum\limits_{j=1}^{n}B_{1j}(m) q^m -w_1{n_N}\sum\limits_{m>0} B_{1j}(m)q^m \equiv 3\sum\limits_{m>0} \sum\limits_{j=1}^{n}B_{1j}(m) q^m \pmod{n_N}. \]
Since we mentioned that $3 \nmid N-1$, then $3\nmid{n_N}$ and so 
$(n_N,3)=1$. Multiplying both sides by the inverse of $3 \pmod{n_N}$ leaves us with a cusp form with integer coefficients 
congruent to the normalized Eisenstein series (minus its constant term) modulo $n_N$.\\

\noindent The remaining cases, $N=7,11 \pmod{12}$ follow similarly.\\

Note that, for $N=7\pmod{12}$, %$N-1=6\pmod{12}$ and so $N-1=12k+6$ so $12 \nmid N-1$ 
we have $(N-1,12)=6$ and we use \[g(\tau)=\left(2\sum\limits_{j=1}^n f_{1j}\right)-w_1f_{11} n_N . \] Again we will need that $2$ is invertible modulo $n_N$ and since in this case $n_N=\frac{N-1}{6}=2k+1$, we see that $(n_N,2)=1$ so $2$ is invertible modulo $n_N$.\\

And finally, for $N=11\pmod{12}$, %$N-1=10\pmod{12}$ and so $N-1=12k+10=2(6k+5)$ so $3,4,6,12 \nmid N-1$ so
we have $(N-1,12)=2$ and we use \[g(\tau)=\left(6\sum\limits_{j=1}^n f_{1j}\right)-w_1f_{11} n_N. \] Lastly we will need that $6$ is invertible modulo $n_N$ and since in this case $n_N=\frac{N-1}{2}=6k+5$, we see that $(n_N,6)=1$ so $6$ is invertible modulo $n_N$.
\end{proof}

%First we'll construct a cusp form which is $0 \pmod{n_N}$. For this we need to construct a linear combination of the $f_{ij}$ such that the coefficients are $\sum\limits_{j=1}^{n}B_{ij}(m)$ plus some multiple of $\frac{N-1}{12}$. For this we make use of the fact that $\sum_{j=1}^n \frac{1}{w_i}=\frac{N-1}{12}$.

%\begin{equation}\label{eq} \sum\limits_{j=1}^n\left(1+\dfrac{1}{w_j}\right)f_{ij}= \left(\dfrac{1}{2w_j} +\dfrac{1}{12}\right)\left(\dfrac{N-1}{12}\right)+\sum\limits_{m>0} \left[\left(\dfrac{N-1}{12}\right)\sum\limits_{j=1}^n B_{ij}(m)+\sum\limits_{j=1}^n B_{ij}(m)\right] q^m. \end{equation}
%Clearly this expression is congruent to the Eisenstein series $\pmod{n_N}$. Note that $2w_jf_{ij}$ has constant term $1$. So we can multiply $2w_jf_{ij}$  by $-\left(\left(\frac{1}{2w_j} +\frac{1}{12}\right)\left(\frac{N-1}{12}\right)\right)$ and add it to the left hand side of \ref{eq}, this eliminates the constant term, and doesn't change anything modulo $n_N$. The result is a weight $2$ cusp form, congruent to the Eisenstein series $E_{2,N}$ $\pmod{n_N}$.

Now, we have that there exists a cusp form of level $N$ with integral coefficients which is congruent to $\dfrac{1}{24}(NE_2(N\tau)-E_2(\tau))$ modulo $n_N$ except for the constant term. We will use this result in the next proposition.

First we require some notation, let $\ell_N:=\text{num}\left(\dfrac{N^2-1}{24}\right)$. This is the minimum positive integer that clears denominators of $E_{2,N}(\tau)$ because multiplying $E_{2,N}(\tau)$ by $N^2-1$ clears denominators and $24$ is the largest number we can divide by that does not hurt integrality (all coefficients of $E_{2,N}$ except the constant term are divisible by $24$). %In fact, $24$ is the largest that the denominator could be without ruining integrality because the coefficient of $q$ in $E_{2,N}$ is always $24$. 
%Since the constant term of $E_{2,N}$ is $1$, we also required $\ell_N$ to be an integer. 
Note that for $N>3$, $\ell_N=\left(\dfrac{N^2-1}{24}\right)$ and for $N=2,3$, $\ell_N=1$.\\

\begin{remark} Hanson Smith noted that $\ell_N$ is an upper bound for the genus of $X_1(N)$ ($N>3$ prime).
\end{remark}

\vspace{3mm}

\noindent With the notation defined above we prove the following proposition. 

\begin{Proposition} Let $\ell_N$ and $n_N$ be as above. Then $\ell_NE_{2,N}$ has integral coefficients and there exists a cusp form $G_N(\tau)$ of level $N$ and weight $2$ with integral coefficients such that \[\ell_NE_{2,N}(\tau) \equiv -N G_N(\tau) \pmod{n_N}. \]

\end{Proposition}
\begin{proof}
%We first treat the cases when the space of cusp forms of level $N$ are empty. In these (finitely many) cases, we have $n_N=1$. Thus, for these $N$ the congruence above follow from the fact that $\ell_NE_{2,N}(\tau)$ has integral coefficients.
We first treat the cases $N=2,3$, since for these we have $\ell_N=1$. We would like to show that $E_{2,2}(\tau)$ and $E_{2,3}(\tau)$ are congruent to $0 \pmod{n_N}$ because the space of weight $2$ cusp forms of levels $2$ and $3$ are both empty.  Since $n_2=n_3=1$, this follows because the coefficients of $E_{2,2}(\tau)$ and $E_{2,3}(\tau)$ are integers.\\ %\[E_{2,2}(\tau)=\frac{1}{3}(4E_2(2\tau)-E_2(\tau)) \] and \[E_{2,3}(\tau)=\frac{1}{8}(9E_2(2\tau)-E_2(\tau)) \] Since the coefficients of $ 

Now we continue with the remaining cases (and can assume $N>3$). When $N$ is prime, we have the following simplified expression for $E_{2,N}(\tau)$ (cf. \eqref{e2n}): \begin{equation}\label{e2nprime} E_{2,N}(\tau)=\dfrac{1}{(N+1)(N-1)}\left(N^2E_2(N\tau)-E_2(\tau)\right).\end{equation} 
 
%Note that an elementary, case-by-case analysis shows that $(N-1,12)(N+1,12) \mid 24$. From this, and the fact that $\ell_N$ is always itself an integer, we have that $\ell_N E_{2,N}$ has integral coefficients.
 
% By adding and subtracting $NE_2(\tau)$ inside $\left(N^2E_2(N\tau)-E_2(\tau)\right)$, 

\noindent We can rearrange this to get a more convenient expression for $E_{2,N}$ as follows:
 \[ E_{2,N}(\tau)=\dfrac{N}{(N+1)(N-1)}\left(NE_2(N\tau)-E_2(\tau)\right)+ \dfrac{1}{N+1} E_2(\tau).\]
 
\noindent Then we can multiply $E_{2,N}$ by $\ell_N$ to get
 \[ \ell_NE_{2,N}(\tau)=\dfrac{N}{24}\left(NE_2(N\tau)-E_2(\tau)\right)+ \dfrac{N-1}{24} E_2(\tau).\]
 By the congruence in the previous proposition we have that there exists a cusp form $g(\tau)$ with integral coefficients of weight $2$ and level $N$  such that \begin{equation}\label{prevprop}\left(\dfrac{1}{24}\left(NE_2(N\tau)-E_2(\tau)\right)\right) -\dfrac{N-1}{24}= g(\tau)+K(\tau)n_N\end{equation} where $K(\tau)$ is some $q$-series with integer coefficients. We subtract $\dfrac{N-1}{24}$ which is equal to the constant term of $\left(\dfrac{1}{24}\left(NE_2(N\tau)-E_2(\tau)\right)\right)$ so that the left hand side of (\ref{prevprop}) has integer coefficients.
 
 \noindent Thus we have
  \[ \ell_NE_{2,N}(\tau)=N\left(g(\tau)+\dfrac{N-1}{24}+K(\tau)n_N\right)+ \dfrac{N-1}{24} E_2(\tau).\]
  
\noindent We define $G_N(\tau):=-g(\tau)$, distribute $N$, and combine constant terms. Then we can write: 
  \[ \ell_NE_{2,N}(\tau)=-NG_N(\tau)+\dfrac{N^2-1}{24}+NK(\tau)n_N+ \dfrac{N-1}{24} (E_2(\tau)-1).\]
  
\noindent Since $\dfrac{N^2-1}{24}$ is an integer multiple of $n_N$ it is $0\pmod{n_N}$. Also, since $N-1$ is a multiple of $n_N$ and each coefficient of $\dfrac{E_2(\tau)-1}{24}$ is an integer, each coefficient of $\dfrac{N-1}{24} (E_2(\tau)-1)$ is $0\pmod{n_N}$, and similarly for $NK(\tau)n_N$. So we have shown that $\ell_NE_{2,N}(\tau) \equiv -N G_N(\tau) \pmod{n_N}. $
 \end{proof}

Now that we have shown this congruence between $\ell_NE_{2,N}(\tau)$ and  $-N G_N(\tau)$ modulo $n_N$, we have that for $N$ prime, the expression $\dfrac{\ell_N E_{2,N}(\tau)+NG_N(\tau)}{n_N}$ has integer coefficients. We can call these our \begin{equation}\label{QN} Q_N^{(N)}(\tau):=\frac{-\ell_N E_{2,N}(\tau)-NG_N(\tau)}{n_N},\end{equation} and we define \begin{equation}\label{Q1} Q_1^{(N)}(\tau):=\frac{-\ell_N}{n_N} E_{2,1}(\tau).\end{equation} Note that $Q_1^{(N)}(\tau)$ also has integer coefficients and that $E_{2,1}(\tau)=E_2(\tau)$.

\noindent These functions (\ref{QN}) and (\ref{Q1}) will be our trace functions for the $\mathbb{Z}/N\mathbb{Z}$ module. In order to prove the existence of this module we require the following lemma:

\begin{Lemma}\label{congruencelem} Let $N$ be prime. Then \[Q_N^{(N)}(\tau) \equiv Q_1^{(N)}(\tau) \pmod{N}. \]
\end{Lemma}

\begin{proof} Again we start with the cases $N=2,3$:\\

\noindent For $N=2$, we have $Q_2^{(2)}(\tau)=-E_{2,2}(\tau)=\frac{-1}{3}(4E_2(2\tau)-E_2(\tau))$ and $Q_1^{(2)}(\tau)=-E_{2}(\tau)$. To show that $Q_2^{(2)}(\tau)\equiv Q_1^{(2)}(\tau) \pmod{2}$, it suffices to show that $3Q_2^{(2)}(\tau)\equiv 3Q_1^{(2)}(\tau) \pmod{2}$ because $(2,3)=1$. We see the latter because $3Q_2^{(2)}(\tau)=-4E_2(2\tau)+E_2(\tau)$, $3Q_1^{(2)}(\tau)=-3E_2(\tau)$, and $-3\pmod{2}=1\pmod{2}$.\\
 
\noindent The case $N=3$ is similar, we have $Q_3^{(3)}(\tau)=E_{2,3}(\tau)=\frac{-1}{8}(9E_2(3\tau)-E_2(\tau))$ and  $Q_1^{(3)}(\tau)=-E_{2}(\tau)$. Again, it suffices to show that $8Q_3^{(3)}(\tau)\equiv 8Q_1^{(3)}(\tau) \pmod{3}$, which can be seen using the fact that $-8\pmod{3}=1\pmod{3}$. \\

\noindent For $N>3$, we will show that \[{Q_N^{(N)}(\tau)}{n_N} \equiv {Q_1^{(N)}(\tau)}{n_N} \pmod{N}. \]  If the above congruence is true then since $(n_N,N)=1$, the congruence in the statement of the lemma will be true. 
Beginning with the left hand side, we have that
%\[{Q_N(\tau)}=\dfrac{1}{24n_N} \left( N^2E_2(N\tau)-E_2(\tau) \right)+ \dfrac{NG_N(\tau)}{n_N},\]
%and so
\[{Q_N^{(N)}(\tau)}{n_N}=\dfrac{-1}{24} \left( N^2E_2(N\tau)-E_2(\tau) \right)- NG_N(\tau).\]
\noindent We can rewrite this as 
\[{Q_N^{(N)}(\tau)}{n_N}=\dfrac{-N^2}{24}E_2(N\tau)+\dfrac{1}{24}E_2(\tau)- NG_N(\tau),\]

\noindent or equivalently 
\[{Q_N^{(N)}(\tau)}{n_N}=\dfrac{-N^2}{24}\left(E_2(N\tau)-1\right) -\dfrac{N^2}{24} +\dfrac{1}{24}\left(E_2(\tau)-1\right) +\dfrac{1}{24}- NG_N(\tau).\]

\noindent Reducing modulo $N$ gives \[Q_N^{(N)}(\tau)n_N \equiv \dfrac{1}{24}(E_2(\tau)-1)-\dfrac{N^2-1}{24} \pmod{N}. \]

\noindent Next we look at the right hand side. We have

\[Q_1^{(N)}(\tau){n_N}=-\ell_N E_{2,1}(\tau)=-\dfrac{N^2-1}{24}E_2(\tau)=-\dfrac{N^2}{24}E_2(\tau)+ \dfrac{1}{24}E_2(\tau)\]
which is equal to
\[-\dfrac{N^2}{24}\left(E_2(\tau)-1\right) + \dfrac{1}{24}\left(E_2(\tau)-1\right) -\dfrac{N^2-1}{24}.\]
Reducing modulo $N$ we get
\[Q_1^{(N)}(\tau){n_N}\equiv \dfrac{1}{24}\left(E_2(\tau)-1\right) -\dfrac{N^2-1}{24}\pmod{N}.\]

\noindent And thus  \[Q_N^{(N)}(\tau){n_N} \equiv Q_1^{(N)}(\tau){n_N} \pmod{N}. \] \end{proof}
\vspace{3mm}
\noindent For $N$ prime, we give quasimodular weight $2$ trace functions for $\mathbb{Z}/N\mathbb{Z}$ as follows: 
\[f_{g}^{(N)}(\tau):=\begin{dcases}
    Q_1^{(N)}(\tau) & \text{if } g=e, \\
    Q_N^{(N)}(\tau) &\text{if } g\neq e.
\end{dcases} \]

\begin{Theorem}\label{eiscuspmoonshine} There exists a virtual graded $\mathbb{Z}/N\mathbb{Z}$-module $V=\bigoplus\limits_{n}V_n$ such that \[f_{g}^{(N)}(\tau)=\sum\limits_{n=0}^{\infty}\emph{tr}(g \mid V_n)q^n\] where $f_{g}^{(N)}(\tau)$ is the quasimodular form of weight $2$ and level $N$ with integral coefficients as defined as above.
\end{Theorem}
\begin{proof} Showing the existence of this virtual module amounts to showing that the multiplicities of the irreducible representations of $\mathbb{Z}/N\mathbb{Z}$ in the module are integral. Thus we need to show the integrality of $\dfrac{1}{N} \sum\limits_{g\in \mathbb{Z}/N\mathbb{Z}} c_g(n)\overline{\chi_i(g)}$, for each $i$, where $c_g(n)$ are the coefficients of the trace functions for $g\in \mathbb{Z}/N\mathbb{Z}$, and $\chi_i(g)$ are irreducible characters of $\mathbb{Z}/N\mathbb{Z}$. This is equivalent to showing that $\langle \chi_i, f_g^{(N)}(\tau) \rangle $ is integral, for $i \in \{1\dots N\}$. 
For $\chi_1$, the character corresponding to the trivial representation, we have
 \begin{equation}\label{triv}\langle \chi_1, f_g^{(N)}(\tau) \rangle=\dfrac{1}{N}(Q_1^{(N)}(\tau)+(N-1)Q_N^{(N)}(\tau) )=Q_N^{(N)}(\tau)+\dfrac{1}{N}(Q_1^{(N)}(\tau)-Q_N^{(N)}(\tau)).\end{equation}
 
In other words, for integrality of (\ref{triv}) we need that $Q_1^{(N)}(\tau)\equiv Q_N^{(N)}(\tau)\pmod{N}$. This is true by Lemma \ref{congruencelem}.

For all other $\chi_i$, we make use of the fact that, for $\zeta$ a primitive $N$th root of unity, $\zeta+\dots+\zeta^{N-1}=-1$ and in particular that $\zeta^k+\dots+\zeta^{k(N-1)}=-1$ for $1\leq k\leq N-1$.
So all other $\chi_i$'s give  \begin{equation} \langle \chi_i, f \rangle=\dfrac{1}{N}(Q_1^{(N)}(\tau)-Q_N^{(N)}(\tau)), \end{equation}
which is again integral by the congruence $Q_1^{(N)}(\tau)\equiv Q_N^{(N)}(\tau)\pmod{N}$.
\end{proof}

%Extending Zagier's formula to other $H_g(\tau)$ sometimes the expression also involves weight $2$ cusp forms. If we rearrange terms we get the following,
\section{Arithmetic/geometric connections}\label{Section4}
In this section, we describe some arithmetic connections between the trace functions of the $M_{23}$-module given in the previous section with expressions in (\ref{eiscuspmathieu}) and the $\mathbb{F}_p$ point counts on (Jacobians of) modular curves. The expressions for $E_{2,N}(\tau)$ are not always integral on their own, and in the levels with cusp forms we saw that adding a multiple (with specific denominator) of a cusp form to $E_{2,N}(\tau)$ gives an expression with integral coefficients. The choices of cusp forms we used in the previous section (given explicitly in the appendix) are such that we get integral coefficients when we add $\frac{N}{n_{N}} G_N$ to $E_{2,N}(\tau)$.\\

 For $M_{23}$ this is summarized below:
\begin{center} 
\begin{enumerate}[(a)]
\item []
\item$ -2E_{2,{11}} (\tau)+{\dfrac{11}{5}}G_{11}(\tau) $has integral coefficients.
\vspace{1mm}
\item  $-2E_{2,{14}}(\tau)+{\dfrac{14}{3}}G_{14}(\tau) $ has integral coefficients.
\vspace{1mm}

\item $ -2E_{{2,15}}(\tau)+{\dfrac{15}{4}}G_{15}(\tau)$ has integral coefficients.
\vspace{1mm}

%\item  $-2E_{{2, 23}}(\tau)+\mathbf{\dfrac{23}{11}}G_{23a}(\tau)+\mathbf{\dfrac{69}{11}}G_{23b}(\tau)
\item $-2E_{{2, 23}}(\tau)+{\dfrac{23}{11}}G_{23}(\tau)$ has integral coefficients.

\end{enumerate}
\end{center}

\noindent Because coefficients of weight $2$ cusp forms admit a certain geometric interpretation, these expressions give divisor conditions on the number of $\mathbb{F}_p$ points on Jacobians of modular curves. Let $J_{0}(N)$ denote the Jacobian of the modular curve $X_0(N)$. For $N=11,14,15$, we have that $J_0(N)$ is an elliptic curve. For $N=23$ $J_0(N)$ is an abelian surface, namely, the Jacobian of a genus $2$ curve.

For simplicity, here, we restrict our attention to $N$ such that $X_0(N)$ are elliptic curves. %In this case, we have that $ \# J_0(N)(\mathbb{F}_p)=  \# X_0(N)(\mathbb{F}_p)$. 

The result below relies on the relationship between cusp forms and point counts on elliptic curves. For an introductory reference for elliptic curves, their $\mathbb{F}_p$ points, and the relationship to cusp forms of weight $2$, see \cite{milne} (in particular, we apply Theorem $7.10$ of \textit{loc.~cit.}).

\noindent \begin{Corollary}\label{divisibility} We have the following (known) divisibility conditions arising from $M_{23}$ 
\begin{enumerate} 
 \item  For $p\neq 11$, we have $5\mid \# J_0(11)(\mathbb{F}_p)$. %and we have $ \# J_0(11)(\mathbb{F}_{11}) \equiv 1 \pmod{5}$.
 \item For $p\neq 2,7$ we have $3\mid \# J_0(14)(\mathbb{F}_p)$. %and we have $ \# J_0(14)(\mathbb{F}_{2}) \equiv 2 \pmod{3}$ and $ \# J_0(14)(\mathbb{F}_{7}) \equiv 1 \pmod{3}$.
 \item For $p\neq 3,5$, we have $4\mid \# J_0(15)(\mathbb{F}_p)$. %and we have $\#J_0(15)(\mathbb{F}_3)\equiv 1 \pmod{4}$ and $\#J_0(15)(\mathbb{F}_5)\equiv 3 \pmod{4}$.
 
\end{enumerate}

\end{Corollary}
 \begin{proof} (1) We have that \[\left(\dfrac{-121}{60}E_2(11\tau)+\dfrac{1}{60}E_2(\tau)\right)+{\dfrac{11}{5}}G_{11}(\tau)\in \mathbb{Z}[\![q]\!]. \]  Using the definition of $E_2(\tau)$, we see that this is equal to \[-2+\dfrac{242}{5}\sum_{m=1}^{\infty} \sigma(m)q^{11m}-\dfrac{2}{5}\sum\limits_{n=1}^{\infty} \sigma(n)q^n+ \dfrac{11}{5}G_{11}(\tau).\]
 Note that we defined $G_{11}(\tau)=2\eta^2(\tau)\eta^2(11\tau)$ so all of its coefficients are even, and in fact $G_{11}(\tau)=2\widetilde{G}_{11}(\tau)$ where the $\widetilde{G}_{11}(\tau)=\sum\limits_{n>0} c_{11}(n)$ is the normalized cusp form whose coefficients correspond to the number of $\mathbb{F}_p$ points on the Jacobian  of $X_0(11)$, denoted $\# J_0(11)(\mathbb{F}_p)$. The correspondence is given as follows:
 \begin{equation}\label{essential}c_{11}(p)=p+1-\#J_0(11)(\mathbb{F}_p).\end{equation}
 
We can see that if $11\nmid n$, then the $n$th coefficient of (a) is $-\dfrac{2}{5}\sigma(n)+\dfrac{22}{5}c_{11}(n) \in \mathbb{Z}$. The integrality of the coefficient of $p\neq 11$ then implies that $-2(p+1)+22c_{11}(p)\equiv 0 \pmod{5}$. Substituting $c_{11}(p)$ for the right hand side of \eqref{essential} gives us that $\# J_0(11)(\mathbb{F}_p)\equiv 0 \pmod{5}$.

We note that for $p=11$, we can directly compute $\# J_0(11)(\mathbb{F}_{11})$ with the fact that $c_{11}(11)=1$. Thus we have $\# J_0(11)(\mathbb{F}_{11})=11.$

(2) If $2,7\nmid n$, then the $n$th coefficient of (b) is $\dfrac{1}{3}\sigma(n)+\dfrac{14}{3}c_{14}(n) \in \mathbb{Z}$. The integrality of the coefficient of $p\neq 2,7$ then implies that $p+1+14c_{14}(p)\equiv 0 \pmod{3}.$ We substitute $c_{14}(p)$ for the right hand side of the following:
\begin{equation}\label{essential14}c_{14}(p)=p+1-\#J_0(14)(\mathbb{F}_p),\end{equation} 
and this gives us that $\# J_0(14)(\mathbb{F}_p)\equiv 0 \pmod{3}$.

We note that for $p=2$ and $p=7$, we can directly compute $\# J_0(14)(\mathbb{F}_2)$ and $\# J_0(14)(\mathbb{F}_7)$ with the facts that $c_{14}(2)=-1$ and $c_{14}(7)=1$. Thus we have $\# J_0(14)(\mathbb{F}_2)=4$ and $\# J_0(14)(\mathbb{F}_7)=7.$

(3) If $3,5\nmid n$, then the $n$th coefficient of (c) is $\dfrac{1}{4}\sigma(n)+\dfrac{15}{4}c_{15}(n) \in \mathbb{Z}$. The integrality of the coefficient of $p\neq 3,5$ then implies that $p+1+15c_{15}(p)\equiv 0 \pmod{4}$ We substitute $c_{15}(p)$ for the right hand side of the following:
\begin{equation}\label{essential15}c_{15}(p)=p+1-\#J_0(15)(\mathbb{F}_p),\end{equation} 
and this gives us that $\# J_0(15)(\mathbb{F}_p)\equiv 0 \pmod{4}$.

We note that for $p=3$ and $p=5$, we can directly compute $\# J_0(15)(\mathbb{F}_3)$ and $\# J_0(15)(\mathbb{F}_5)$ with the facts that $c_{15}(3)=-1$ and $c_{15}(5)=1$. Thus we have $\# J_0(15)(\mathbb{F}_3)=5$ and $\# J_0(15)(\mathbb{F}_5)=5.$
\end{proof}
\vspace{3mm}
We expect infinitely many more such expressions (as in Corollary \ref{divisibility}) arising from the trace functions of the $\mathbb{Z}/N\mathbb{Z}$-modules of Theorem \ref{eiscuspmoonshine}. In the cases above, since the modular curves are elliptic curves and therefore (isomorphic to) their own Jacobians, any divisibility conditions on the number of $\mathbb{F}_p$ points on the Jacobians are equivalent to divisibility conditions on the number of $\mathbb{F}_p$ points on the modular curves themselves. In general, for any prime $N$, the integrality of trace functions from these $\mathbb{Z}/N\mathbb{Z}$-modules are equivalent to divisibility conditions on the number of $\mathbb{F}_p$ points on the Jacobians  of $X_0(N)$ (cf. Theorem $7.10$ of \cite{milne}).

\begin{remark} The integrality conditions used in the proof or Corollary \ref{divisibility} implied congruences of the form $-2(p+1)+22c_{11}(p)\equiv 0 \pmod{5}$. These are equivalent to statements such as $(p+1)\equiv 0\pmod{5}$ iff $c_{11}(p)\equiv 0 \pmod{5}$. Jeffrey Lagarias noted that if one reframes these statements as, for example, \[ p=4\pmod{5} \text{ if and only if } 5\mid c_{11}(p),\] %where $c_{11}(p)$ is the coefficient of the cusp form $\widetilde{G}_{11}(\tau)$,
the expressions can then be written in the following way resembling the Ramanujan congruences:
 \[ \text{For } p=5n+4 \text{ we have } c_{11}(p)=0\pmod{5}.\]
\end{remark}
\vspace{1mm}
\begin{remark} In this formulation of the trace functions of the $M_{23}$-module (\ref{eiscuspmathieu}), we made a choice with the multiple of the cusp form. The denominator is fixed but the choices we made of the numerator are not unique. In fact, we could add any multiple of $NG_N(\tau)$ and still satisfy the congruences necessary for Mathieu moonshine. Therefore, the module here is one in an infinite family of possible modules one can consider. A similar statement holds for the modules of Theorem \ref{eiscuspmoonshine}. It would be interesting to see if stronger results about point counts on modular Jacobians might be obtained by studying these families as a whole.
\end{remark}

%\begin{remark}The congruences can be modified if you know more about the divisibility of $m_i(p)$.  It turns out that there are some $p$ for which $\gcd(m_1(p), \dots, m_{17}(p))>1$. For example, take $p=599$, then $\gcd(m_1(599), \dots, m_{17}(599))=8$. If we look at $8\mid m_i(p)$, the congruence above becomes $6 \#E_{11}(p)\equiv 5\#E_{14}(p) \pmod{20}$. Looking at other modules by varying the multiple of the cusp form results in more variations of this expression.
%\end{remark}

\section{An explicit module construction}\label{Section5}
In the previous sections, the trace functions in Theorem \ref{eiscuspmoonshine} and those of the $M_{23}$ module (\ref{eiscuspmathieu}) have involved cusp forms. The results of this were some arithmetic/geometric observations in addition to the minimality of the constant ($\frac{\ell_N}{n_N}$) in front of $E_{2,N}(\tau)$ that guarantees integrality of trace functions' coefficients. If we restrict our functions to only involve Eisenstein series, we remove the cusp form contribution (thus we can no longer divide by $n_N$). We define an alternative set of trace functions for $\mathbb{Z}/N\mathbb{Z}$ in this way.
\noindent For $N$ prime, we give purely Eisenstein quasimodular weight $2$ trace functions for a $\mathbb{Z}/N\mathbb{Z}$-module as follows: 
\[F_{g}^{(N)}(\tau):=\begin{dcases}
   - \ell_NE_{2}(\tau) & \text{if } g=e, \\
   -\ell_NE_{2,N}(\tau) &\text{if } g\neq e.
\end{dcases} \]

 It can be easily seen from the methods in Section \ref{Section3} that there also exists a module for which these are the trace functions. Indeed, we will construct such a module explicitly in this section.
 
 Although the modules with quasimodular trace functions involving cusp forms were arguably more interesting, the advantage of purely Eisenstein quasimodular trace functions is that we can actually give a construction of the module in terms of vertex operator algebras. Moreover, when the space of cusp forms of level $N$ is empty, the purely Eisenstein trace functions are equal to the trace functions in Theorem \ref{eiscuspmoonshine}. Given this, it would be interesting to see if the method presented here may be modified so as to obtain the modules that do include cusp from contributions, discussed in Section \ref{Section3}.
 
To construct the purely Eisenstein modules we will find a vertex operator algebra that has the $F_{g}^{(N)}(\tau)$ as their trace functions. For this we will use two  Heisenberg vertex algebras and a Clifford module vertex algebra. 
 First, note that for $N>3$ we have $\ell_N=\dfrac{N^2-1}{24}$ and 
 \begin{align*}
 q\dfrac{d}{dq} \log \left(\dfrac{1}{\eta^{(N^2-1)}(\tau)}\right)&= -\dfrac{N^2-1}{24} E_2(\tau), \text{ and}\\
 q\dfrac{d}{dq} \log\left(\dfrac{\eta(\tau)}{\eta^N(N\tau)}\right)&= -\dfrac{N^2-1}{24} E_{2,N}(\tau).
 \end{align*}
 For the remaining cases $N=2,3$, we have that $\ell_N=1$.\\
When $N=2$ we have 
  \begin{align*}
 q\dfrac{d}{dq} \log  \left(\dfrac{1}{\eta^{24}(\tau)}\right)&= - E_2(\tau), \text{ and}\\
q\dfrac{d}{dq} \log \left(\dfrac{\eta^8(\tau)}{\eta^{16}(2\tau)}\right)&= - E_{2,2}(\tau). 
\end{align*}
And when $N=3$ we have
\begin{align*}
 q\dfrac{d}{dq} \log  \left(\dfrac{1}{\eta^{24}(\tau)}\right)&= - E_2(\tau), \text{ and}\\
q\dfrac{d}{dq} \log \left(\dfrac{\eta^3(\tau)}{\eta^{9}(3\tau)}\right)&= - E_{2,3}(\tau).
 \end{align*}
In what follows, we describe the module construction. %for the cases $N>3$ and explain how the construction is modified for the cases $N=2,3$. At the end, we give the module explicitly for all $N$.

%We note that analogues of \eqref{thateq} and Lemmas \ref{tr1}, \ref{tr2}, and \ref{clifflem}, can be found for the $N=2,3$ cases in exactly the same way using the logarithmic derivatives listed above. 

Let $D$ denote the derivative $D(\cdot):=q\frac{d}{dq}(\cdot)$. From the above equations we see that finding a module whose trace functions are equal to $F_{g}^{(N)}(\tau)$ with $N>3$ is equivalent to finding a module whose trace functions are equal to:
\begin{equation}\label{thateq1} \begin{dcases} D\left(\dfrac{1}{\eta^{(N^2-1)}(\tau)}\right)\eta^{N^2}(\tau)\dfrac{1}{\eta(\tau)} & \text{if } g=e, \\
D\left(\dfrac{\eta(\tau)}{\eta^N(N\tau)}\right)\eta^N(N\tau)\dfrac{1}{\eta(\tau)} & \text{if } g\neq e.
 \end{dcases}\end{equation}
 
Similarly, we note that finding a module whose trace functions are  $F_{g}^{(2)}(\tau)$ and  $F_{g}^{(3)}(\tau)$ are equivalent to finding a module whose trace functions are equal to:
 
 \begin{equation}\label{thateq2} \begin{dcases} D\left(\dfrac{1}{\eta^{24}(\tau)}\right)\eta^{32}(\tau)\dfrac{1}{\eta^8(\tau)} & \text{if } g=e, \\
D\left(\dfrac{\eta^8(\tau)}{\eta^{16}(2\tau)}\right)\eta^{16}(2\tau)\dfrac{1}{\eta^8(\tau)} & \text{if } g\neq e,
 \end{dcases}\end{equation}
 and
 \begin{equation}\label{thateq3} \begin{dcases} D\left(\dfrac{1}{\eta^{24}(\tau)}\right)\eta^{27}(\tau)\dfrac{1}{\eta^3(\tau)} & \text{if } g=e, \\
D\left(\dfrac{\eta^3(\tau)}{\eta^9(3\tau)}\right)\eta^9(3\tau)\dfrac{1}{\eta^3(\tau)} & \text{if } g\neq e,
 \end{dcases}\end{equation}
respectively.
The next two lemmas indicate how to recover the first of the three factors in each of \cref{thateq1,thateq2,thateq3}. To formulate it, let $\mathfrak{h}=\mathbb{C}^{24\ell_N}$ and let $\gamma^{}$ be an automorphism of order $N$ of $\mathfrak{h}$ such that its characteristic polynomial is $\text{char}_{\gamma^{}}(x)=\dfrac{(x^N-1)^N}{(x-1)}$ (For $N=2,3$, take the characteristic polynomials $\text{char}_{\gamma^{}}(x)=\dfrac{(x^2-1)^{16}}{(x-1)^8}$ and $\text{char}_{\gamma^{}}(x)=\dfrac{(x^3-1)^9}{(x-1)^3}$, respectively). 
 Let $V^{}:= S(b_i(-n) \mid n>0; \hspace{1mm} i=1,\dots, 24\ell_N)$ (where $S(x_1, x_2 \dots):= S(\oplus_{i=1}^{\infty} \mathbb{C}x_i)$) be the Heisenberg vertex algebra for $\mathfrak{h}$ with non-degenerate symmetric bilinear form $\langle \cdot, \cdot \rangle$ fixed by $\gamma$. The action of $\gamma$ on $\mathfrak{h}$ extends naturally to $V$. See \cite{FrenkelBenzvi} for more information on the Heisenberg vertex algebra construction. We denote the $L(0)$ operator for this Heisenberg vertex algebra by $L_1(0)$ and let $c_1$ be its central charge.  
 \begin{Lemma}\label{tr1} For $N>3$, we have graded trace functions for $V^{}$ as follows:
 \[\emph{tr}\left(q^{L_1(0)-\frac{c_1}{24}}\mid V^{}\right)=\dfrac{1}{\eta^{(N^2-1)}(\tau)}\]
 and 
 \[\emph{tr}\left(\gamma q^{L_1(0)-\frac{c_1}{24}}\mid V^{}\right)=\dfrac{\eta(\tau)}{\eta^N(N\tau)}.\]
 \end{Lemma}
 
\noindent Since we instead need the derivatives of those functions, we can take the traces as follows:
\begin{Lemma}\label{tr2} For $N>3$, we have \begin{align*}
\emph{tr}\left(\left(L_1(0)-\frac{c_1}{24}\right)q^{L_1(0)-\frac{c_1}{24}}\mid V^{}\right)&=D\left(\dfrac{1}{\eta^{(N^2-1)}(\tau)}\right) \\ 
 \emph{tr}\left(\left(L_1(0)-\frac{c_1}{24}\right)\gamma q^{L_1(0)-\frac{c_1}{24}}\mid V^{}\right)&= D\left(\dfrac{\eta(\tau)}{\eta^N(N\tau)}\right). 
 \end{align*}
 \end{Lemma}

\noindent Note that the two equations above correspond to the first factor of the $g=e$ and $g\neq e$ functions in \eqref{thateq1} and the analogous statements hold for $N=2,3$ (\cref{thateq2,thateq3}). 

Next, for the cases $N>3$, in order to recover the second factor in each of the functions in \cref{thateq1}, % we multiply the trace functions in Lemma \ref{tr2} 
 we need a module with trace functions $\eta^{N^2}(\tau)$ and $\eta(N\tau)^{N}$, respectively. This can be done using a Clifford module vertex algebra. For this construction we follow Duncan and Harvey \cite{DH}. In this setting, let $\mathfrak{p}$ be a one dimensional complex vector space with a symmetric bilinear form. Let $a(r):= a\otimes t^r$ for $a \in \mathfrak{p}$. Let ${\hat{\mathfrak{p}}}_{}=\mathfrak{p}[t,t^{-1}]t^{1/2}$ and $\hat{\mathfrak{p}}_{\text{tw}}=\mathfrak{p}[t,t^{-1}]$ with the bilinear form extended so that $\langle a(r), b(s) \rangle=\langle a,b \rangle \delta_{r+s,0}$.\\
 
 We define $\text{Cliff}(\mathfrak{p})$ to be the Clifford algebra attached to $\mathfrak{p}$. Let $\hat{\mathfrak{p}}^+:=\mathfrak{p}[t]t^{1/2}$, where $\langle \hat{\mathfrak{p}}^+ \rangle$ is a subalgebra of the Clifford algebra and let $\mathbb{C}v$ be a $\langle \hat{\mathfrak{p}} ^+\rangle$ module such that that $1v=v$ and $p(r)v=0$ for $r>0$.
 Then we define
 \[A(\mathfrak{p})_{}^{}:=\text{Cliff}(\hat{\mathfrak{p}} )\otimes_{\langle \hat{\mathfrak{p}}^{+}\rangle} \mathbb{C}v,\]
 and $A(\mathfrak{p})$ has the structure of a super vertex operator algebra with Virasoro element
 \[ \omega:=%\dfrac{-1}{4}\sum\limits_{i=1}^{24\ell_N}
  \mathfrak{p}(-3/2)\mathfrak{p}(-1/2)v. \] 
 
Let $\text{Cliff}(\hat{\mathfrak{p}}_{\text{tw}})$ be the Clifford algebra attached to $\hat{\mathfrak{p}}_{\text{tw}}$, let $v_{\text{tw}}$ be such that $1v_{\text{tw}}=v_{\text{tw}}$, and let $a(r)v_{\text{tw}}=0$ for $a\in \mathfrak{p}$ and $r>0$. Then take $p\in \mathfrak{p}$ such that $\langle p,p \rangle=-2$ and $p(0)^2=1$. Define $v_{\text{tw}}^{+}:=(1+p(0))v_{\text{tw}}$ so that $p(0)v_{\text{tw}}^{+}=v_{\text{tw}}^{+}$ and let $\hat{\mathfrak{p}}_{\text{tw}}^{>}:=\mathfrak{p}[t]t$.  Then we define  \[A(\mathfrak{p})_{\text{tw}}^{+}:=\text{Cliff}(\hat{\mathfrak{p}}_{\text{tw}} )\otimes_{\langle \hat{\mathfrak{p}}_{\text{tw}}^{>}\rangle} \mathbb{C}v_{\text{tw}}^{+},\] 

\noindent so that $A(\mathfrak{p})_{\text{tw}}^+$ is isomorphic to $\bigwedge(p(-n)\mid n>0)v_{\text{tw}}^+$ (where $\bigwedge(x_1, x_2 \dots):= \bigwedge(\oplus_{i=1}^{\infty} \mathbb{C}x_i)$).

By the reconstruction theorem described in \cite{FrenkelBenzvi} we can see that $A(\mathfrak{p})_{\text{tw}}$ is a twisted module for $A(\mathfrak{p})$ with fields $Y_{\text{tw}}: A(\mathfrak{p}) \otimes A(\mathfrak{p})_{\text{tw}} \to A(\mathfrak{p})_{\text{tw}}(\!(z^{1/2})\!)$ where \[ Y_{\text{tw}}(u(-1/2)v,z)=\sum\limits_{n \in \mathbb{Z}} u(n)z^{-n-1/2}\] with $u\in \mathfrak{p}$.  Since $A(\mathfrak{p})_{\text{tw}}^+$ is a submodule of $A(\mathfrak{p})_{\text{tw}}$ (generated by $v^+_{\text{tw}}$), it can be verified that $A(\mathfrak{p})_{\text{tw}}^+$ is a twisted module for $A(\mathfrak{p})$ so that the above map can be restricted to $A(\mathfrak{p})^+_{\text{tw}}$.

Let $L_2(0)$ be the $L(0)$ operator for the Clifford module vertex algebra and $c_2$ its central charge. Then we can see that $\text{tr}\left(p(0)q^{L_2(0)-\frac{c_2}{24}}\mid A(\mathfrak{p})_{\text{tw}}^+\right)= \eta(\tau)$.  We would like a module with graded dimension equal to $\eta(\tau)^{N^2}$ so we will consider a tensor product of these $ A(\mathfrak{p})_{\text{tw}}^+$ (we have from \cite{FLH} that the tensor product of vertex algebras is naturally a vertex algebra).\\

 To do this, we define 
\[ \widetilde{A}({\mathfrak{p}})^{}:= {A}({\mathfrak{p}}_1) \otimes \cdots \otimes {A}(\mathfrak{p}_{N^2})  \]
and
\[ \widetilde{A}({\mathfrak{p}})^{}_{\text{tw}}:= {A}({\mathfrak{p}_{1} })_{\text{tw}}^++\otimes \cdots \otimes {A}(\mathfrak{p}_{N^2} )_{\text{tw}}^+  \]
where each $A(\mathfrak{p}_i)_{\text{tw}}^+$ is isomorphic to $\bigwedge(p_i(-n) \mid n>0)v_{\text{tw}}^+$. Then we can define

\[\widetilde{Y}_{\text{tw}}: \widetilde{A}(\mathfrak{p})^{} \otimes \widetilde{A}(\mathfrak{p})^{+}_{\text{tw}} \to \widetilde{A}(\mathfrak{p})^{+}_{\text{tw}}(\!(z^{1/2})\!)\] where \begin{align*} \widetilde{Y}_{\text{tw}}((u(-1/2)v_1\otimes \cdots \otimes u(-1/2)v_{N^2})v,z)&=Y_1(u_1(-1/2)v_1,z) \otimes \cdots \otimes Y_{N^2}(u_{N^2}(-1/2)v_{N^2},z)\\&=\sum\limits_{n \in \mathbb{Z}^{N^2}} u_1(n_1)\otimes \cdots \otimes u_{N^2}(n_{N^2})z^{-n_1 \cdots -n_{N^2} -\frac{N^2}{2}}, \end{align*}
with $n^{}=(n_1, \dots, n_{N^2})$,
and finally  \[ \widetilde{p}^{}(0):=p_1(0)\otimes \cdots \otimes p_{N^2}(0). \]

Let $\sigma^{}_{}$ act on $\widetilde{A}({\mathfrak{p}})^{}_{\text{tw}}$ by permuting tensor factors with cycle shape $N^N$.
\noindent Now we have in the next lemma the second factor of each equation in \eqref{thateq1}.
\begin{Lemma}\label{clifflem}  \begin{align*}
\emph{tr}\left(\widetilde{p}^{}(0)q^{L_2(0)-\frac{c_2}{24}}\mid \widetilde{A}(\mathfrak{p})^{}_{\emph{tw}}\right)&= \eta^{N^2}(\tau),\\
\emph{tr}\left(\sigma^{}_{}\widetilde{p}^{}(0)q^{L_2(0)-\frac{c_2}{24}}\mid \widetilde{A}(\mathfrak{p})^{}_{\emph{tw}}\right)&= \eta^N(N\tau).
\end{align*}
\end{Lemma}
\vspace{3mm}
When $N=2,3$, the construction is similar, but we define $\widetilde{A}({\mathfrak{p}})^{}$, $\widetilde{A}({\mathfrak{p}})^{}_{\text{tw}}$, and $\widetilde{p}^{}(0)$ to have $32$ (resp. $27$) tensor factors (instead of $N^2$) and we take instead permutations $\sigma^{}_{}$ with cycle shape $2^{16}$ (resp. $\sigma^{}_{}$ with cycle shape $3^{9}$).\\

\noindent Then for $N=2$, with $\widetilde{A}({\mathfrak{p}})^{}_{\text{tw}}:= {A}({\mathfrak{p}_{1} })_{\text{tw}}^+\otimes \cdots \otimes {A}(\mathfrak{p}_{32} )_{\text{tw}}^+$, we have:
\begin{align*}
\text{tr}\left(\widetilde{p}^{}(0)q^{L_2(0)-\frac{c_2}{24}}\mid \widetilde{A}(\mathfrak{p})^{}_{\text{tw}}\right)&= \eta^{32}(\tau),\\
\text{tr}\left(\sigma^{}_{}\widetilde{p}^{}(0)q^{L_2(0)-\frac{c_2}{24}}\mid \widetilde{A}(\mathfrak{p})^{}_{\text{tw}}\right)&= \eta^{16}(2\tau).
\end{align*}
\noindent And for $N=3$ and $\widetilde{A}({\mathfrak{p}})^{}_{\text{tw}}:= {A}({\mathfrak{p}_{1} })_{\text{tw}}^+\otimes \cdots \otimes {A}(\mathfrak{p}_{27} )_{\text{tw}}^+$, we have:
\begin{align*}
\text{tr}\left(\widetilde{p}^{}(0)q^{L_2(0)-\frac{c_2}{24}}\mid \widetilde{A}(\mathfrak{p})^{}_{\text{tw}}\right)&= \eta^{27}(\tau),\\
\text{tr}\left(\sigma^{}_{}\widetilde{p}^{}(0)q^{L_2(0)-\frac{c_2}{24}}\mid \widetilde{A}(\mathfrak{p})^{}_{\text{tw}}\right)&= \eta^{9}(3\tau).
\end{align*}

%For the second equation in Lemma \ref{clifflem} we make sense of the action of $\sigma$ on $\widetilde{A}(\mathfrak{p})_{\text{tw}}$ by diagonalizing its action on the space $\mathbb{C}^{24\ell_N}$. We get eigenvalues $\alpha_1 \cdots \alpha_{24\ell_N}$ and can regard our $24\ell_N$ copies of $\mathfrak{p} \simeq \mathbb{C}$ from the Clifford module construction as the corresponding eigenspaces.  We let $\sigma$ act as multiplication by $\alpha_i$ on $p_i(-n)$ in the $i$-th factor $A(\mathfrak{p}_{i})_{\text{tw}}^+$.  The eigenvalues $\alpha_i$ are $N$th roots of unity dictated by the characteristic polynomial of $\sigma$.\\

 Lastly, we recover the third factor in \cref{thateq1,thateq2,thateq3}.
%We define \[ U^{(N)}:=\begin{dcases}
 %  S(\hat{\mathbb{C}})= S(b(-n) \mid n>0)  & \text{  if  } N>3 \\
 %S(\hat{\mathbb{C}}^{8})= S(b_i(-n) \mid n>0; \hspace{1mm} i=1,\dots, 8) & \text{  if  } N=2 \\
  %S(\hat{\mathbb{C}}^{3})= S(b_i(-n) \mid n>0; \hspace{1mm} i=1,\dots, 3) & \text{  if  } N=3 \\
% \end{dcases}\]
 %We define $\mathcal{H}:=\begin{dcases}
% \mathbb{C} & \text{ if } N>3 \\
 % \mathbb{C}^8 & \text{ if } N=2 \\
 %\mathbb{C}^3 & \text{ if } N=3 \\
% \end{dcases}$  
For this we define $\mathfrak{k}$ to be $\mathbb{C}$ when $N>3$, $\mathbb{C}^8$ when $N=2$, and $\mathbb{C}^3$ when $N=3$.
 Let $U:=S(k(-n) \mid n>0)$ (suitably modified when $N=2,3$) be the Heisenberg vertex algebra for $\mathfrak{k}$. We denote the $L(0)$ operator for this Heisenberg vertex algebra by $L_3(0)$ and let $c_3$ be its central charge. 
Then we have the following lemma

\begin{Lemma}\label{heis2lem} When $N>3$, we have 
\[\emph{tr}\left(q^{L_3(0)-\frac{c_3}{24}}\mid U\right)=\dfrac{1}{\eta(\tau)}.\]
\end{Lemma}
\noindent And note that for $N=2,3$ we get graded dimension equal to $\dfrac{1}{\eta^8(\tau)}$ and $\dfrac{1}{\eta^3(\tau)}$, respectively.

To get a vertex algebra whose trace function is the desired product in \cref{thateq1,thateq2,thateq3}, we let \[W^{(N)}:=V^{} \otimes \widetilde{A}(\mathfrak{p})^{}_{} \otimes U.\]
We take the following canonically twisted module for the vertex algebra $W^{(N)}$:
\[W^{(N)}_{\text{tw}}:=V^{} \otimes \widetilde{A}(\mathfrak{p})_{\text{tw}}^{} \otimes U.\]

%We can extend the permutation $\sigma^{(N)}_{}$ for $N>3$ to have cycle shape $1^{N^2-1}N^N1^1$, $\sigma^{(2)}_{}$ to have cycle shape $1^{24}2^{16}1^8$, and $\sigma^{(3)}_{}$ to have cycle shape $1^{24}3^{9}1^3$. 
The actions of $\gamma$ and $\sigma$ naturally extend to $W^{(N)}$ and $W^{(N)}_{\text{tw}}$, by letting them act trivially on the factors where they have not already been defined (i.e. $\gamma$ acts as $\gamma\otimes \text{id}\otimes \text{id}$ and $\sigma$ acts as $\text{id}\otimes \sigma\otimes \text{id}$). Note that with this definition both $\gamma$ and $\sigma$ are automorphisms of $W^{(N)}$, and act equivariantly on the twisted module in the sense that we have $Y_{\text{tw}}(gu,z)gv=gY_{\text{tw}}(u,z)v$ for $u\in W^{(N)}$ and $v\in W^{(N)}_{\text{tw}}$, and $g$ equal to $\gamma$ or $\sigma$.

Lastly, we define the operator $L(0):=L_1(0)+L_2(0)+L_3(0)$ and the central charge $c:=c_1+c_2+c_3$ where the subscripts indicate the operators and central charges for the Heisenberg and Clifford vertex algebras described above.
Then we have that the following forms are equal to the forms in \eqref{thateq1} when $N>3$, \eqref{thateq2} when $N=2$, and \eqref{thateq3} when $N=3$:
\[ \begin{dcases} \text{tr}\left(\widetilde{p}^{}(0)\left(L_1(0)-\frac{c_1}{24}\right)q^{L(0)-\frac{c}{24}}\mid W^{(N)}_{\text{tw}}\right)  & \text{if } g=e, \\
  \text{tr}\left(\gamma^{} \sigma^{}\widetilde{p}^{}(0)\left(L_1(0)-\frac{c_1}{24}\right)q^{L(0)-\frac{c}{24}}\mid W^{(N)}_{\text{tw}}\right)   & \text{if } g\neq e.
 \end{dcases}\]

 Thus we have constructed a vertex algebra for the $\mathbb{Z}/ N \mathbb{Z}$-module with purely Eisenstein quasimodular trace functions.
\begin{Theorem} Let $N$ be prime. Then $W^{(N)}_{\emph{tw}}=\bigoplus\limits_n W^{(N)}_{\emph{tw},n}$ is an infinite dimensional virtual graded module for $\mathbb{Z}/ N \mathbb{Z}$ such that 
\[F_{g}^{(N)}(\tau)=\sum\limits_{n=0}^{\infty} \emph{tr}(g \mid W^{(N)}_{\emph{tw},n})q^n.\]
\end{Theorem}
\appendix
\section{Cusp forms}
We give the cusp forms relevant for $M_{23}$ explicitly below:
\begin{align*}
G_{11}(\tau)&=2\eta^2(\tau)\eta^2(11\tau), \\
G_{14}(\tau)&=\eta(\tau)\eta(2\tau)\eta(7\tau)\eta(14\tau),\\
G_{15}(\tau)&=\eta(\tau)\eta(3\tau)\eta(5\tau)\eta(15\tau),\\
G_{23a}(\tau)&=\dfrac{\eta^3(\tau)\eta^3(23\tau)}{\eta(2\tau)\eta(46\tau)}+3\eta^2(\tau)\eta^2(23\tau)+4\eta(\tau)\eta(2\tau)\eta(23\tau)\eta(46\tau)+4\eta^2(2\tau)\eta^2(46\tau), \\
G_{23b}(\tau)&=\eta^2(\tau)\eta^2(23\tau), \\
G_{23}(\tau)&=G_{23a}(\tau)+3G_{23b}(\tau).
\end{align*}

\bibliography{QMbibtex}{}

\begin{thebibliography}{10}

\bibitem{AnagiannisChengHarrison}
Vassilis Anagiannis, Miranda Cheng, and Sarah M.~Harrison.
\newblock K3 elliptic genus and an umbral moonshine module.
\newblock 09 2017.

\bibitem{Borcherds}
Richard~E. Borcherds.
\newblock Monstrous moonshine and monstrous {L}ie superalgebras.
\newblock {\em Invent. Math.}, 109(2):405--444, 1992.

\bibitem{Kensbook}
Kathrin Bringmann, Amanda Folsom, Ken Ono, and Larry Rolen.
\newblock {\em Harmonic {M}aass forms and mock modular forms: theory and
  applications}, volume~64 of {\em American Mathematical Society Colloquium
  Publications}.
\newblock American Mathematical Society, Providence, RI, 2017.

\bibitem{CDH}
M.~C.~N. Cheng, J.~F.~R. Duncan, and J.~A. Harvey.
\newblock Umbral moonshine and the {N}iemeier lattices.
\newblock {\em Res. Math. Sci.}, 1:Art. 3, 81, 2014.

\bibitem{CDMeroJacobi}
Miranda Cheng and John Duncan.
\newblock Meromorphic jacobi forms of half-integral index and umbral moonshine
  modules.
\newblock 07 2017.

\bibitem{2}
Miranda C.~N. Cheng.
\newblock {$K3$} surfaces, {$N=4$} dyons and the {M}athieu group {$M_{24}$}.
\newblock {\em Commun. Number Theory Phys.}, 4(4):623--657, 2010.

\bibitem{holographic}
Miranda C.~N. Cheng and John F.~R. Duncan.
\newblock On {R}ademacher sums, the largest {M}athieu group and the holographic
  modularity of moonshine.
\newblock {\em Commun. Number Theory Phys.}, 6(3):697--758, 2012.

\bibitem{ConwayNorton}
J.~H. Conway and S.~P. Norton.
\newblock Monstrous moonshine.
\newblock {\em Bull. London Math. Soc.}, 11(3):308--339, 1979.

\bibitem{DMZ}
Atish Dabholkar, Sameer Murthy, and Don Zagier.
\newblock Quantum black holes, wall crossing, and mock modular forms.
\newblock {\em arXiv:1208.4074, to appear in Cambridge Monographs in
  Mathematical Physics}, 08 2012.

\bibitem{DGO}
John F.~R. Duncan, Michael~J. Griffin, and Ken Ono.
\newblock Proof of the umbral moonshine conjecture.
\newblock {\em Research in the Mathematical Sciences}, 2:26, 2015.

\bibitem{DH}
John F.~R. Duncan and Jeffrey~A. Harvey.
\newblock The umbral moonshine module for the unique unimodular {N}iemeier root
  system.
\newblock {\em Algebra Number Theory}, 11(3):505--535, 2017.

\bibitem{DO}
John F.~R. Duncan and Andrew O'Desky.
\newblock Super vertex algebras, meromorphic jacobi forms and umbral moonshine.
\newblock {\em Journal of Algebra}, 515:389 -- 407, 2018.

\bibitem{5}
Tohru Eguchi and Kazuhiro Hikami.
\newblock Note on twisted elliptic genus of {$K3$} surface.
\newblock {\em Phys. Lett. B}, 694(4-5):446--455, 2011.

\bibitem{EOT}
Tohru Eguchi, Hirosi Ooguri, and Yuji Tachikawa.
\newblock Notes on the {$K3$} surface and the {M}athieu group {$M_{24}$}.
\newblock {\em Exp. Math.}, 20(1):91--96, 2011.

\bibitem{FrenkelBenzvi}
Edward Frenkel and David Ben-Zvi.
\newblock {\em Vertex algebras and algebraic curves}, volume~88 of {\em
  Mathematical Surveys and Monographs}.
\newblock American Mathematical Society, Providence, RI, second edition, 2004.

\bibitem{FLH}
Igor~B. Frenkel, Yi-Zhi Huang, and James Lepowsky.
\newblock On axiomatic approaches to vertex operator algebras and modules.
\newblock {\em Mem. Amer. Math. Soc.}, 104(494):viii+64, 1993.

\bibitem{FLM}
Igor~B. Frenkel, James Lepowsky, and Arne Meurman.
\newblock A moonshine module for the {M}onster.
\newblock In {\em Vertex operators in mathematics and physics ({B}erkeley,
  {C}alif., 1983)}, volume~3 of {\em Math. Sci. Res. Inst. Publ.}, pages
  231--273. Springer, New York, 1985.

\bibitem{4}
Matthias~R. Gaberdiel, Stefan Hohenegger, and Roberto Volpato.
\newblock Mathieu {M}oonshine in the elliptic genus of {$K3$}.
\newblock {\em J. High Energy Phys.}, (10):062, 24, 2010.

\bibitem{3}
Matthias~R. Gaberdiel, Stefan Hohenegger, and Roberto Volpato.
\newblock Mathieu twining characters for {$K3$}.
\newblock {\em J. High Energy Phys.}, (9):058, 20, 2010.

\bibitem{Gannon}
Terry Gannon.
\newblock Much ado about {M}athieu.
\newblock {\em Adv. Math.}, 301:322--358, 2016.

\bibitem{Gross}
Benedict~H. Gross.
\newblock Heights and the special values of {$L$}-series.
\newblock In {\em Number theory ({M}ontreal, {Q}ue., 1985)}, volume~7 of {\em
  CMS Conf. Proc.}, pages 115--187. Amer. Math. Soc., Providence, RI, 1987.

\bibitem{Katz}
Nicholas~M. Katz.
\newblock Galois properties of torsion points on abelian varieties.
\newblock {\em Invent. Math.}, 62(3):481--502, 1981.

\bibitem{Martin}
Kimball Martin.
\newblock The {J}acquet-{L}anglands correspondence, {E}isenstein congruences,
  and integral {$L$}-values in weight 2.
\newblock {\em Math. Res. Lett.}, 24(6):1775--1795, 2017.

\bibitem{Mazur}
B.~Mazur.
\newblock Modular curves and the {E}isenstein ideal.
\newblock {\em Inst. Hautes \'Etudes Sci. Publ. Math.}, (47):33--186 (1978),
  1977.

\bibitem{Michael}
Michael~H. Mertens.
\newblock Eichler-{S}elberg type identities for mixed mock modular forms.
\newblock {\em Adv. Math.}, 301:359--382, 2016.

\bibitem{milne}
J.~S. Milne.
\newblock {\em Elliptic curves}.
\newblock BookSurge Publishers, Charleston, SC, 2006.

\bibitem{Quat}
Patricia~L. Quattrini.
\newblock The effect of torsion on the distribution of {$\Sha$} among quadratic
  twists of an elliptic curve.
\newblock {\em J. Number Theory}, 131(2):195--211, 2011.

\bibitem{Thompson}
J.~G. Thompson.
\newblock Finite groups and modular functions.
\newblock {\em Bull. London Math. Soc.}, 11(3):347--351, 1979.

\bibitem{MR554402}
J.~G. Thompson.
\newblock Some numerology between the {F}ischer-{G}riess {M}onster and the
  elliptic modular function.
\newblock {\em Bull. London Math. Soc.}, 11(3):352--353, 1979.

\end{thebibliography}
\bibliographystyle{plain}

\end{document}